\documentclass[12pt]{article}
\usepackage[top=2.54cm, bottom=2.54cm, left=2.80cm, right=2.80cm]{geometry}
\usepackage[tbtags]{amsmath}
\usepackage{amssymb}
\usepackage{amsthm}
\usepackage{fancyhdr}
\usepackage{latexsym}
\usepackage{mathrsfs}
\usepackage{wasysym}
\usepackage{fancyhdr}
\usepackage{float}
\usepackage{graphicx}
\usepackage[numbers,sort&compress]{natbib}
\usepackage{setspace}
\allowdisplaybreaks[4]

\newtheorem{theorem}{\bf Theorem}[section]
\newtheorem{lemma}[theorem]{Lemma}
\newtheorem{Def}[theorem]{Definition}

\newtheorem{corollary}[theorem]{Corollary}

\begin{document}
\begin{spacing}{1.1}
\title{New results on $k$-independence of hypergraphs}
\author{Lei Zhang$^a$\thanks {Email addresses: shuxuezhanglei@163.com(L. Zhang).anchang@fzu.edu.cn(An Chang)}, \, An Chang$^a$\\
$^a$Center for Discrete Mathematics and Theoretical Computer Science,\\
 Fuzhou University, Fuzhou, Fujian, P.R. China}
\date{}
\maketitle
\begin{abstract}
Let $H=(V,E)$ be an $s$-uniform hypergraph of order $n$ and $k\geq 0$ be an integer. A $k$-independent set $S\subseteq H$ is a set of vertices such that the maximum degree in the hypergraph induced by $S$ is at most $k$. Denoted by $\alpha_k(H)$ the maximum cardinality of the $k$-independent set of $H$. In this paper, we first give a lower bound of $\alpha_k(H)$ by the maximum degree of $H$. Furthermore, we prove that $\alpha_k(H)\geq \frac{s(k+1)n}{2d+s(k+1)}$ where $d$ is average degree of $H$, and $k\geq 0$ is an integer.
\end{abstract}

{\bf AMS}\,: 05C65; 05C69


\section{Introduction}

\noindent

Let $G=(V,E)$ be a graph on $n$ vertices and $k\geq 0$ be an integer. A $k$-independent set $S\subseteq V$ is a set of vertices such that the maximum degree in the graph induced by $S$ is at most $k$. With $\alpha_k(G)$ we denote the maximum cardinality of a $k$-independent set of $G$ and it is called the $k$-independence number of $G$. In particular, $\alpha_0(G)=\alpha(G)$ is the usual independence number of $G$.

In recent years, as a generalization of the independence number of graphs, the $k$-independence number of graphs attracted more attention of researchers. The first result on bounding $k$-independence number was given by Y. Caro and Z. Tuza in \cite{Caro1991}, where it was shown that if the average degree $d\geq k+1$ then $\alpha_k(G)\geq \frac{k+2}{2(d+1)}n$. Furthermore, Caro and Hansberg \cite{Caro2013} showed that if the average degree $d\geq 0$ then $\alpha_k(G)\geq \frac{(k+1)n}{\lceil d\rceil+k+1}n$. Recently, Shimon Kogan prove that $\alpha_k(G)\geq \frac{(k+1)n}{d+k+1}n$ for general $k$ in \cite{Kogan2017}, and improve the previous best bound obtained by Caro and Hansberg.

A hypergraph $H$ is a pair $(V,E)$, where $E\in P(V)$ and $P(V)$ stands for the power set of $V$. The elements of $V=V(G)$, labeled as $[n]=\{1,\cdots,n\}$, are referred to as vertices and the elements of $E=E(H)$ are called edges. A hypergraph $H$ is said to be $s$-uniform for an integer $s\geq2$ if, for all $e\in E(H)$, $|e|=s$. Obviously, a 2-uniform hypergraph is just an ordinary graph. Motivated by above study on the $k$-independence number of graphs, researchers pay much attention to the independence number of hypergraphs. Some relevant results are presented as following.

In 1991, Yair Caro and Zsolt Tuza firstly estimate the maximum cardinality $\alpha(H)$ of an independent set in an $s$-uniform hypergraph $H$.

\begin{theorem}[\cite{Caro1991}]
Let $H=(V,E)$ be an $s$-uniform hypergraph with $s\geq 2$. Then
\[\alpha(H)\geq \sum \limits_{v\in V}f(d(v)),\]
where $d(v)$ is the degree of $v$, i.e., the number of edges containing $v$, and the function $f$ is given by
\[f(d):=\prod \limits_{i=1}^{d} (1-\frac{1}{i(s-1)+1}).\]
\end{theorem}

In 1999, Torsten Thiele gives a lower bound of the independence number for general hypergraphs.

\begin{theorem}[\cite{Thiele1999}]
Let $H=(V;E)$ be a hypergraph of rank $r$. Then
\[\alpha(H)\geq \sum \limits_{v\in V}f_r(d(v)),\]
where $d(v)$ is the degree of $v$, i.e., the number of edges containing $v$, and the function $f_r: N^{r}_0\rightarrow \mathcal{R}$ is given by
\[f_r(d)=\sum \limits_{i\in N^{r}_0} \left[\prod \binom{d_m}{i_m}\right]\frac{(-1)^{\sum i_m}}{\sum(m-1)\cdot i_m+1}.\]
\end{theorem}

In 2012, B\'{e}la Csaba, Thomas A. Plick, Ali Shokoufandeh give another lower bound of the independence number for an $s$-uniform hypergraph, and improve the Spencer's result\cite{Spencer1972}.

\begin{theorem}[\cite{Thomas2012}]
Let $H$ be an $s$-uniform hypergraph for $s\geq 3$, and let $d_1,d_2,\cdots,d_n$ be the degrees of its vertices, with $n>0$. Then
\[\alpha(H)\geq e^{-\frac{\gamma}{s-1}}\cdot \sum \limits_{i=1}^n \frac{1}{(d_i+1)^{\frac{1}{s-1}}}.\]
where $\gamma=0.5772\ldots$ is the Euler-Mascheroni constant.
\end{theorem}

A generalization for the $k$-independence number of $s$-uniform hypergraphs was obtained by Caro and Tuza\cite{Caro1991}, which improved earlier results of Favaron \cite{Favaron1988}. (Here, ``1-independent" means ``independent".)

\begin{theorem}[\cite{Caro1991}]
For every natural number $k$ and every $s$-uniform hypergraph $H$ with vertex set $V$,
\[\alpha_k(H)\geq \sum \limits_{x\in V} f_{k,s}(d(x)).\]
where $d(x)$ is the degree of $x$, i.e., the number of edges containing $x$, and the function $f_{k,s}(i)$ is given by
\[
   f_{k,s}(i)=\left\{
      \begin{array}{ll}
      1-\frac{i}{ks}, \ \ \ \ \ \ \ \ \ \ \ \ \ \ \ \ \ \ \ \ \ 0\leq i\leq k;\\
      \prod \limits_{j=1}^{i-k+1} (1-\frac{1}{j(s-1)+1}), \ \ \ \ \ \ \ i\geq k,
   \end {array}
   \right.
\]
for all nonnegative integers $k\geq 1,s\geq 2$, and $i\geq 0$.
\end{theorem}

In this paper, we study the $k$-independence set of $s$-uniform hypergraphs. First, some necessary notations are presented in section 2. Then, in section 3, for an $s$-uniform hypergraph $H$ of order $n$, we first give a lower bound of the $k$-independent number $\alpha_k(H)$ by the maximum degree. Further more, we prove that $\alpha_k(H)\geq \frac{s(k+1)}{s(k+1)+2d}n$, where $d$ is average degree of $H$ and $k\geq 0$ is an integer. More generally we obtain that $\alpha_k(H)\geq f(\frac{2d}{s(k+1)})n$, where \[f(x)=\frac{1}{1+x}(1+\frac{\{x\}(1-\{x\})}{(\lfloor x\rfloor+1)(\lfloor x \rfloor+2)}).\]

\section{Preliminary}

\noindent

A hypergraph $H$ is a pair $H=(V,E)$ where $V$ is a set of elements, and $E$ is a set of non-empty subsets of $V$.
Therefore, $E$ is a subset of $\mathcal{P}(V)\backslash\{\emptyset\}$, where $\mathcal{P}(V)$ is the power set of $V$.
The elements of $V$ are referred as vertices, and $n(H)=|V(H)|$ denotes its cardinality, while the elements of $E$ are called edges,
and $e(H)$ denotes its cardinality. A hypergraph is said to be $s$-uniform if each edge contains precisely $s$ vertices in $E$, where $s$ is an positive integer and $s\geq 2$. For a vertex $v\in V(H)$, $deg(v)=deg_{H}(v)$ is the degree of $v$ in $H$, that is the number of edges contain $v$. We denote the maximum degree of $H$ by $\Delta(H)$ and the average degree $\frac{1}{n(H)}\sum \limits_{v\in V(H)}deg(v)$ by $d(H)$. For a vertex $v\in V(H)$, $H-v$ represents the hypergraph $H$ without vertex $v$ and all the edges incident to $v$.

We redefine the induced subhypergraph of a hypergraph as follows.

\begin{Def}
Let $H=(V,E)$ is a hypergraph, where $V=\{v_1,v_2,\cdots,v_n\}$, $E=\{e_1,e_2,\cdots,e_m\}$. For any subset $S\subseteq V$
we call the hypergraph $H[S]=(S,E')$ a subhypergraph induced by the set $S$ if $E'$ consists of all those subsets in $E$ that are
completely contained in $S$. And $deg_{S}(v)$ stands for the degree $deg_{H[S]}(v)$ of $v$ in $H[S]$.
\end{Def}

Based on the above definition, we give the definition of the $k$-independent set for hypergraphs.

\begin{Def}
Let $H=(V,E)$ be a hypergraph on $n$ vertices and $k\geq 0$ be an integer. A $k$-independent set $S\subseteq H$ is a set of vertices such that the maximum degree in the hypergraph induced by $S$ is at most $k$.
\end{Def}
Let $\alpha_k(H)$ denote the maximum cardinality of the $k$-independent set of $H$ and called it the $k$-independence number of $H$. For $k=0$ we have $\alpha_0(H)=\alpha(H)$, where $\alpha(H)$ is the independence number of $H$.

\section{Main results}

\noindent

In this section, we will give two lower bounds for the $k$-independent number of $s$-uniform hypergraphs by the maximum degree or average degree.

\begin{Def}
For an $s$-uniform hypergraph $H$, we denote by $\chi_k(H)$ the $k$-chromatic number of $H$, i.e. the minimum number $t$ such that there is a partition $V(H)=V_1(H)\cup V_2(H)\cup \cdots \cup V_t(H)$ of the vertex set $H$ such that $\Delta(H[V_i])\leq k$ for all $1\leq i\leq t$.
\end{Def}
For a set $J\subset \{1,2,\cdots,m\}$, we call the family $H'=(e_j/j\in J)$ the partial hypergraph generated by the set $J$. For $x\in X$, define the star $H(x)$ with centre $x$ to be the partial hypergraph formed by the edges containing $x$. We called a \emph{$\beta$-star} of a vertex $x$ a family $H^{\beta}(x)\subset H(x)$ such that

(i)$E\in H^{\beta}(x)\Rightarrow |E|\geq 2$.

(ii)$E,E'\in H^{\beta}(x)\Rightarrow E\cap E'=\{x\}$.

We call the \emph{$\beta$-degree} of a vertex $x$ the largest number of edges of a $\beta$-star of $x$. We denote by $d_{H}^{\beta}(x)$ the $\beta$-degree of $x$, by $\Delta^{\beta}(H)=\max \limits_{x\in X} d_{H}^{\beta}(x)$ the maximum $\beta$-degree. The following result is obtained by Lov\'{a}sz in 1968.

\begin{corollary}[\cite{L1968}]
For every hypergraph $H$ of maximum $\beta$-degree $\Delta^{\beta}$, we have $\chi(H)\leq \Delta^{\beta}(H)+1$.
\end{corollary}

Then we can obtain an upper bound for the $k$-chromatic number of $s$-uniform hypergraphs.

\begin{theorem}
If $H$ is an $s$-uniform hypergraph of maximum degree $\Delta$, then $\chi_k(H)\leq \lceil \frac{\Delta}{k}\rceil$.
\end{theorem}

\begin{proof}
For an $s$-uniform hypergraph $H$ on $V(H)$ of maximum degree $\Delta$, let $H'$ be the hypergraph on $V(H)$ whose edges are the subgraphs of maximum degree $k$. Then by Definition 3.1 and Corollary 3.2, we have
\[\chi_k(H)=\chi(H')\leq \Delta^{\beta}(H')+1 \leq [\frac{\Delta}{k}]+1=\lceil \frac{\Delta}{k}\rceil\]
where $[x]$ is the integer-valued function, $\lceil x\rceil$ is the ceiling function.
\end{proof}

Now since $\alpha_k(H)\geq \frac{n}{\chi_k(H)}$, we get a lower bound for the $k$-independent number of $s$-uniform hypergraphs by the maximum degree.

\begin{theorem}
Let $H$ be an $s$-uniform hypergraph of order $n$ and maximum degree $\Delta$. Then
\[\alpha_k(H)\geq \frac{n}{\lceil \frac{\Delta}{k}\rceil}\]
\end{theorem}

Next, to give another lower bound for the $k$-independent number of $s$-uniform hypergraphs by the average degree, we need more preparations.

\begin{Def}
Define the function $f(x)$ for real $x\geq 0$ in the following manner:
\[f(x)=\frac{1}{1+x}(1+\frac{\{x\}(1-\{x\})}{(\lfloor x\rfloor+1)(\lfloor x \rfloor+2)}),\]
where $\lfloor x\rfloor$ is the floor function and $\{x\}$ is the fractional part function.
\end{Def}

\begin{lemma}[\cite{Kogan2017}]
Let $x\geq 0$ be a real number. Then $f(x)\geq \frac{1}{1+x}$, and equality holds if and only if $x$ is an integer.
\end{lemma}

Then our second result can be presented as following.

\begin{theorem}
Let $k\geq 0$ be an integer. Then for any s-uniform hypergraph $H$ of order $n$ and average degree $d$, we have
\[\alpha_k(H)\geq f(\frac{2d}{s(k+1)})n.\]
\end{theorem}

\begin{corollary}
Let $k\geq 0$ be an integer. Then for any s-uniform hypergraph $H$ of order $n$ and average degree $d$, we have
\[\alpha_k(H)\geq \frac{s(k+1)}{s(k+1)+2d}n.\]
\end{corollary}

\begin{proof}
This follows from Theorem 3.7 and Lemma 3.6, as
\[\frac{1}{1+\frac{2d}{s(k+1)}}=\frac{s(k+1)}{s(k+1)+2d}.\]
\end{proof}
Before proving Theorem 3.7 we will need a few more lemmas and definitions.

\begin{Def}
Define a function $g(x)$ for real $x\geq 0$ in the following manner:
\[
   g(x)=\left\{
      \begin{array}{ll}
      1, \ \ \ \ \ \ \ \ \ \ if \ x=0;\\
      \frac{2\lceil x\rceil-x}{\lceil x\rceil(1+\lceil x\rceil)}, \ if \ x>0.
   \end {array}
   \right.
\]
where $\lceil x\rceil$ is the ceiling function.
\end{Def}

By the result of \cite{Kogan2017}, we easily have the following two lemmas:

\begin{lemma}[\cite{Kogan2017}]
For all real $x>0$ we have $f(x)=g(x)$.
\end{lemma}

\begin{lemma}[\cite{Kogan2017}]
The function $f(x)$ is continuous, monotonically decreasing and convex on the interval $[0,\infty)$.
\end{lemma}

\begin{lemma}
Let $k\geq 0$ and $r\geq 0$ be integers. Let $H$ is an s-uniform hypergraph of order $n$ with $e$ edges and average degree $d=\frac{se}{n}$. If $\frac{s}{2}r(k+1)<d\leq \frac{s}{2}(r+1)(k+1)$ holds and
\[\alpha_k(H)\geq \frac{2}{r+2}(n-\frac{e}{(r+1)(k+1)}),\]
then $\alpha_k(H)\geq f(\frac{2d}{s(k+1)})n$.
\end{lemma}

\begin{proof}
Set $t=\frac{2d}{s(k+1)}$. Since $\frac{s}{2}r(k+1)< d\leq \frac{s}{2}(r+1)(k+1)$, we have $r<t\leq r+1$. Hence $\lceil t\rceil=r+1$. Thus we have
\begin{eqnarray*}
\alpha_k(H)&\geq& \frac{2}{r+2}(n-\frac{e}{(r+1)(k+1)})\\
&=&\frac{2}{\lceil t\rceil+1}(n-\frac{dn}{s\lceil t\rceil(k+1)}) \ \ \ \ \ (as \ r+1=\lceil t\rceil)\\
&=&\frac{2}{\lceil t\rceil+1}(n-\frac{tn}{2\lceil t\rceil}) \ \ \ \ \ (as \ t=\frac{2d}{s(k+1)})\\
&=&\frac{2n}{\lceil t\rceil+1}(1-\frac{t}{2\lceil t\rceil})\\
&=&n\frac{2\lceil t\rceil-t}{\lceil t\rceil(\lceil t\rceil+1)}=g(t)n=f(t)n \ \ \ \ \ (by \ Lemma \ 3.10)
\end{eqnarray*}
and we are done.
\end{proof}

\begin{lemma}
Let $k\geq 0$ be an integer. If $H$ is an s-uniform hypergraph of order $n$ with $e$ edges, then
\[\alpha_k(H)\geq n-\frac{e}{k+1}.\]
\end{lemma}

\begin{proof}
Set $H_0=H$. If there is a vertex $v_0\in V(H_0)$ such that $deg_{H_0}(v_0)\geq k+1$, then remove it from the hypergraph $H_0$ and call the resulting hypergraph $H_1$, that is, $H_1=H_0-v_0$.

Now, if there is a vertex $v_1\in V(H_1)$ such that $deg_{H_1}(v_1)\geq k+1$, then remove it from the hypergraph $H_1$ and call the resulting hypergraph $H_2$, that is, $H_2=H_1-v_1$. We can repeat this operation iteratively until we get a hypergraph $H_i$ for some $i>0$ such that the maximum degree of $H_i$ satisfies $\Delta(H_i)\leq k$. Notice that $i\leq \lfloor \frac{e}{k+1}\rfloor$, as there are $e$ edges in $H_0$, and in each iteration the number of edges in the resulting hypergraph is decreased by at least $k+1$. Hence if we have reached iteration $t$ for $t=\lfloor \frac{e}{k+1}\rfloor$, then hypergraph $H_t$ will contain at most $e-(k+1) \lfloor\frac{e}{k+1}\rfloor \leq k$ edges, and thus $H_t$ is a $k$-independent set, which means that $i\leq t$. Since $V(H_i)$ is a $k$-independent set in $H$ and $|V(H_i)|=n-i \geq n-\frac{e}{k+1}$, we are done.
\end{proof}

\begin{corollary}
Let $k\geq 0$ be an integer. If $H$ is an $s$-uniform hypergraph of order $n$ with average degree $0<d\leq
\frac{s}{2}(k+1)$, then
\[\alpha_k(H)\geq f(\frac{2d}{s(k+1)})n.\]
\end{corollary}

\begin{proof}
This follows from Lemma 3.13 by setting $r=0$ in Lemma 3.12.
\end{proof}

\begin{lemma}
Let $H$ be an $s$-uniform hypergraph on $n$ vertices with average degree $d(H)\leq d$ and such that $d+2p+1$ does not divide $n$. Then there is an $s$-uniform hypergraph $H'$ such that $d+2p+1$ divides $m=n(H')$, $d(H')=d(H)\leq d$ and $\frac{\alpha_k(H')}{m}=\frac{\alpha_k(H)}{n}$.
\end{lemma}

\begin{proof}
Let $H'=(d+2p+1)H$. Then $m=n(H')=(d+2p+1)n$ is multiple of $(d+2p+1)$, $d(H')=d(H)$ and $\frac{\alpha_k(H')}{m}=\frac{(d+2p+1)\alpha_k(H)}{(d+2p+1)n}=\frac{\alpha_k(H)}{n}$.
\end{proof}

Finally we are on the spot to give the proof of our main result of this section.

\textbf{Proof of Theorem 3.7:}

Let $k\geq 0$ be an integer. Recall that we need to prove that for any $s$-uniform hypergraph $H$ of order $n$ and average degree $d$, we have
\[\alpha_k(H)\geq f(\frac{2d}{s(k+1)})n.\]
We will prove by induction on integer $r\geq 0$ that for any $s$-uniform hypergraph $H$ of order $n$ and average degree $d \leq \frac{s}{2}(r+1)(k+1)$ we have:
\[\alpha_k(H)\geq f(\frac{2d}{s(k+1)})n.\]
When $r=0$, the result was verified in Corollary 3.14. Assume that the claim holds for $r-1>0$ and we will prove it for $r$.

By Lemma 3.12 and the induction hypothesis, it suffices to prove that if $H$ is an $s$-uniform hypergraph on $n$ vertices, $e$ edges and average degree $d$ satisfying $\frac{s}{2}r(k+1)<d\leq \frac{s}{2}(r+1)(k+1)$, then
\[\alpha_k(H)\geq \frac{2}{r+2}(n-\frac{e}{(r+1)(k+1)}).\]
We can assume that both $n$ and $e$ are divisible by $(r+2)(k+1)$. This is because we can build an $s$-uniform hypergraph $H_0=(r+2)(k+1)H$, namely, $H_0$ is a disjoint union of $(r+2)(k+1)$ copies of $H$, where $d(H)
=d(H_0)$, and the number of vertices $n_0$ and number of edges $e_0$ of $H_0$ are both divisible by $(r+2)(k+1)$.
If $\alpha_k(H')\geq n'f(\frac{2d}{s(k+1)})$ then the original hypergraph $H$ satisfies
\[\alpha_k(H)\frac{n'}{(r+2)(k+1)}f(\frac{2d}{s(k+1)})=nf(\frac{2d}{s(k+1)})\]

We define parameter $t$ as follows:
\[t=\frac{s\cdot e-n\cdot \frac{s}{2}r(k+1)}{\frac{s}{2}(r+2)(k+1)}\]

Because $n$ and $e$ are divisible by $(r+2)(k+1)$, and $d>\frac{s}{2}r(k+1)$, we know that $t$ is an integer and $t>0$.

Set $H_0=H$. If there is a vertex $v_0\in V(H_0)$ such that $deg_{H_0}(v_0)\geq \frac{s}{2}(r+1)(k+1)$, then remove it from the hypergraph $H_0$, and denote the resulting hypergraph by $H_1$, that is, $H_1=H_0-v_0$.

Now if $t>1$ and there is a vertex $v_1\in V(H_1)$ such that $deg_{H_1}(v_1)\geq \frac{s}{2}(r+1)(k+1)$, then remove it from the hypergraph $H_1$ and denote the resulting hypergraph by $H_2$, that is, $H_2=H_1-v_1$.

We repeat this operation iteratively, that is on iteration $i$(starting with $i=0$) we first check if $i=t$ or $\Delta(H_i)<\frac{s}{2}(r+1)(k+1)$, and if one of these conditions holds we terminate the process. Otherwise, we pick a vertex $v_i\in V(H_i)$ such that $deg_{H_i}(v_i)\geq \frac{s}{2}(r+1)(k+1)$ and remove it from the
hypergraph $H_i$. The resulting hypergraph is denoted by $H_{i+1}$, that is, $H_{i+1}=H_i-v_i$.

Suppose that the process above terminated on iteration $j\leq t$, that is, the last hypergraph created in the process is $H_j$. If $j<t$ then $\Delta(H_j)<\frac{s}{2}(r+1)(k+1)$, and thus by Theorem 3.4, we have $\alpha_k(H_j)\geq \frac{n-t}{r+1}$. Now the resting is to prove that when $j=t$ we also have $\alpha_k(H_j)\geq \frac{n-t}{r+1}$.

First we notice that
\begin{eqnarray*}
n-t&=&n-\frac{s\cdot e-n\cdot \frac{s}{2}\cdot r(k+1)}{\frac{s}{2}(r+2)(k+1)}\\
&=&\frac{(r+2)(k+1)n+nr(k+1)-2e}{(r+2)(k+1)}\\
&=&\frac{2[n(r+1)(k+1)-e]}{(r+2)(k+1)} \ \ \ \ \ \ \ \ \ \ \ \ \ \ \ \ \ \ \ \ \ \ \ \ \ \ \ \ \ \ \ \ \ \ \ \ \ \ \ \ \ \ \ \ \ \ \ (3.1)
\end{eqnarray*}
Now we claim that $d(H_t)\leq \frac{s}{2}r(k+1)$. Notice that as in each iteration at least $\frac{s}{2}(r+1)(k+1)$ edges were removed we have that $e(H_t)$, the number of edges in hypergraph $H_t$, satisfies
\begin{eqnarray*}
e(H_t)&\leq &e-t\cdot\frac{s}{2}(r+1)(k+1)\\
&=&e-\frac{2e-nr(k+1)}{(r+2)(k+1)}\cdot \frac{s}{2}(r+1)(k+1)\\
&=&\frac{(r+2)(k+1)e-\frac{s}{2}(r+1)(k+1)[2e-nr(k+1)]}{(r+2)(k+1)}
\end{eqnarray*}
Since $\frac{e(H_t)}{n-t}\leq \frac{1}{2}r(k+1)$, it follows that
\[d(H_t)=\frac{s\cdot e(H_t)}{n-t}=s\cdot \frac{e(H_t)}{n-t}\leq \frac{s}{2}r(k+1).\]
Now since $d(H_t)\leq \frac{s}{2}r(k+1)$, we can apply the induction hypothesis on $H_t$. By the induction hypothesis and Lemma 3.6, we have
\[\alpha_k(H_t)\geq \frac{n-t}{r+1}.\]
We conclude that
\begin{eqnarray*}
\alpha_k(H)&\geq &\alpha_k(H_j)\\
&\geq &\frac{n-t}{r+1}\\
&=&\frac{2}{r+2}(n-\frac{e}{(r+1)(k+1)}) \ \ \ \ (by \ Equation \ 3.1)\\
&\geq &f(\frac{2d}{s(k+1)})n \ \ \ \ \ (by \ Lemma \ 3.12)
\end{eqnarray*}
and this is exactly what we need to prove. This concludes the proof of the induction.

\textbf{Remark.}

For the above two lower bounds of the $k$-independence number of $s$-uniform hypergraphs, as showing Theorem 3.4 and 3.7, respectively. Since
$$s\geq 2, \Delta \geq d,$$
So when $\Delta>k(k+1)$, we have
\[\lceil \frac{\Delta}{k}\rceil>\frac{\Delta}{k} \geq \frac{2d}{s(k+1)}+1.\]
Hence,
\[\frac{n}{\lceil \frac{\Delta}{k}\rceil} < \frac{s(k+1)}{s(k+1)+2d}n\]
Therefore,
we find that when $\Delta>k(k+1)$, the bound in (2) is better than the bound in (1).

\vskip 3mm
\end{spacing}
\end{document}